\documentclass[11pt]{amsart}

\usepackage{epigamath}

\usepackage[english]{babel}

\numberwithin{equation}{section}

\usepackage{enumitem}
\usepackage[OT2,T1]{fontenc}
\usepackage{tikz}
\usepackage{tikz-cd}
\usepackage{mathtools}

\newtheorem{theorem}{Theorem}[section]
\newtheorem{proposition}[theorem]{Proposition}
\newtheorem{lemma}[theorem]{Lemma}
\newtheorem{corollary}[theorem]{Corollary}
\newtheorem{conjecture}[theorem]{Conjecture}
\newtheorem*{conjecture*}{Conjecture}

\theoremstyle{definition}
\newtheorem{definition}[theorem]{Definition}

\theoremstyle{remark}
\newtheorem{remark}[theorem]{Remark}
\newtheorem*{remark*}{Remark}

\newcommand{\Z}{\mathbb{Z}}
\newcommand{\Q}{\mathbb{Q}}
\newcommand{\et}{\mathrm{\acute{e}t}}
\newcommand{\Gm}{\mathbb{G}_{m}}
\DeclareMathOperator{\Spec}{Spec}
\DeclareMathOperator{\Pic}{Pic}
\DeclareMathOperator{\Lie}{Lie}
\DeclareMathOperator{\rank}{rank}
\DeclareMathOperator{\Br}{Br}
\DeclareMathOperator{\NS}{NS}
\newcommand{\ar}{\mathrm{ar}}
\newcommand{\red}{\mathrm{red}}
\newcommand{\tor}{\mathrm{tor}}
\newcommand{\NT}{\mathrm{NT}}
\newcommand{\sep}{\mathrm{sep}}
\newcommand{\To}{\longrightarrow}
\DeclareMathOperator{\ord}{ord}

\DeclareSymbolFont{cyrletters}{OT2}{wncyr}{m}{n}
\DeclareMathSymbol{\Sha}{\mathalpha}{cyrletters}{"58}
\newcommand*{\isoarrow}[1]{\arrow[#1,"\rotatebox{90}{\(\sim\)}"]}

\EpigaVolumeYear{6}{2022} \EpigaArticleNr{7} \ReceivedOn{May 14,
2021}
\InFinalFormOn{December 27, 2021}
\AcceptedOn{January 29, 2022}

\title{The conjectures of Artin--Tate and Birch--Swinnerton-Dyer}
\titlemark{The conjectures of Artin--Tate and Birch--Swinnerton-Dyer}

\author{Stephen Lichtenbaum}
\address{Department of Mathematics, Brown University, Providence, RI 02912}
\email{stephen\_lichtenbaum@brown.edu}

\author{Niranjan Ramachandran}
\address{Department of Mathematics, University of Maryland, College Park, MD 20742 USA.}
\email{atma@math.umd.edu}

\author{Takashi Suzuki}
\address{Department of Mathematics, Chuo University, 1-13-27 Kasuga, Bunkyo-ku, Tokyo 112-8551, Japan}
\email{tsuzuki@gug.math.chuo-u.ac.jp}

\authormark{S.~Lichtenbaum, N.~Ramachandran, and T.~Suzuki}

\AbstractInEnglish{We provide two proofs that the conjecture of Artin--Tate for a fibered surface is equivalent to the conjecture of Birch--Swinnerton-Dyer for the Jacobian of the generic fibre. As a byproduct, we obtain a new proof of a theorem of Geisser relating the orders of the Brauer group and the Tate--Shafarevich group.}

\MSCclass{11G40, 14G10, 19F27}

\KeyWords{Birch--Swinnerton-Dyer conjecture; finite fields; zeta functions; Tate conjecture}

\begin{document}



\maketitle

\begin{prelims}

\DisplayAbstractInEnglish

\bigskip

\DisplayKeyWords

\medskip

\DisplayMSCclass

\end{prelims}


\newpage

\setcounter{tocdepth}{1}

\tableofcontents


\section{Introduction and statement of results}

Let $k=\mathbb F_q$ be a finite field of characteristic $p$ and let $S$ be a smooth projective (geometrically connected) curve over $T=~$Spec$~k$ and let $F = k(S) = \mathbb F_q(S)$ be the function field of $S$. Let $X$ be a smooth proper surface over $T$ with a flat proper morphism $\pi:X \to S$ with smooth geometrically connected generic fiber $X_0$ over Spec~$F$.  The Jacobian $J$ of $X_0$ is an Abelian variety over $F$. 

Our first main result is a proof of the following statement conjectured by Artin and Tate \cite[Conjecture~(d)]{Tate1966}:
\begin{theorem}\label{main} The Artin--Tate conjecture for $X$ is equivalent to the Birch--Swinnerton-Dyer conjecture  for $J$.
\end{theorem}  
Recall that these conjectures concern two (conjecturally finite) groups: the Tate--Shafarevich group $\Sha(J/F)$ of $J$ and the Brauer group $\textrm{Br}(X)$ of $X$. A result of Artin--Grothendieck \cite[Theorem 2.3]{MR528839} \cite[\S 4]{MR244271} is that $\Sha(J/F)$ is finite if and only if $\textrm{Br}(X)$ is finite. 

Our second main result is a new proof of a beautiful result   (\ref{eqn-g}) of Geisser  \cite[Theorem 1.1]{Geisser1} that relates the conjectural finite orders of $\Sha(J/F)$ and $\textrm{Br}(X)$; special cases of  (\ref{eqn-g}) are due to Milne--Gonzales-Aviles  \cite{MR656050, MR1987138}.  

We actually provide two proofs of Theorem \ref{main}; while our first proof  uses  Geisser's result (\ref{eqn-g}), the second (and very short) proof  in \S \ref{suzuki}, completely due to the third-named author, does not.

\subsection{History}  Artin and Tate regarded Theorem \ref{main}  as easier to prove as opposed to the other conjectures  in \cite{Tate1966}. They proved Theorem \ref{main} when $\pi$ is smooth and has a section (\cite[p.427]{Tate1966}) using the equality
\begin{equation}\label{artin} [\Sha(J/F)] = [ \textrm{Br}(X) ]\end{equation} between the orders of the groups $\Sha(J/F)$ and $\textrm{Br}(X)$ which follows from Artin's theorem \cite[Theorem 3.1]{Tate1966}, \cite[Theorem 2.3]{MR528839}: if $\pi$ is generically smooth with connected fibers and admits a section, then $\Sha(J/F) \cong \textrm{Br}(X)$.  Gordon \cite[Theorem 6.1]{MR528839} used (\ref{artin}) to prove Theorem \ref{main} when\footnote{There is another proof (up to $p$-torsion) in this case due to Z.~Yun \cite{Yun}.}  $\pi$ is cohomologically flat with a section (see \cite[Theorem 2.3]{MR528839}). Building on Gordon \cite{MR528839}, Liu--Lorenzini--Raynaud \cite{MR2092767} proved several new cases of Theorem \ref{main} by eliminating the condition of cohomological flatness of $\pi$; their proof  \cite[Theorem 4.3]{MR2092767} proceeds by proving that Theorem \ref{main} is equivalent to a precise relation generalizing (\ref{artin}) between $[\textrm{Br}(X)]$ and $[\Sha(J/F)]$ which in their case had been proved by Milne and Gonzales-Aviles  \cite{MR656050, MR1987138}.  

 As Liu--Lorenzini--Raynaud  (and \href{https://www.jmilne.org/math/articles/1975a.html}{Milne}) point out   \cite[Theorem 2]{MR2125738}, Theorem \ref{main} follows by combining \cite{Tate1966, MR244271, MR0414558, MR2000469}: 
\[AT(X)  \xLeftrightarrow{\ \text{Artin--Tate--Milne}\ } \textrm{Br}(X)~\textrm{finite} \xLeftrightarrow{\ \text{Artin--Grothendieck}\ } \Sha(J/F)~\textrm{finite} \xLeftrightarrow{\ \text{Kato--Trihan}\ } BSD(J).\]  
 In 2018, Geisser pointed out that a slight correction is necessary in the relation \cite[Theorem 4.3]{MR2092767} between $[\textrm{Br}(X)]$ and $[\Sha(J/F)]$; Liu--Lorenzini--Raynaud \cite[Corrected Theorem 4.3]{MR3858404} showed that Theorem \ref{main} holds if and only if this slightly corrected version holds.  This precise relation (Theorem \ref{geisser})  was then proved by Geisser  \cite[Theorem 1.1]{Geisser1} without using Theorem \ref{main}. Thus, combining \cite[Corrected Theorem 4.3]{MR3858404} and  \cite[Theorem 1.1]{Geisser1} gives the second known proof of Theorem \ref{main}.  But this proof relies heavily on the work of Gordon\footnote{Known to have several inaccuracies; see \cite[\S 3.3]{MR3858404}.} \cite{MR528839}  as can be seen from \cite[\S 3, (3.9)]{MR3858404}.

\subsection{Our approach} Our first proof depends on \cite{MR528839} only for the elementary result (\ref{newone}).  As in  \cite{MR528839, MR2092767, MR3858404}, this proof also  follows the strategy in \cite[\S 4]{Tate1966}.  We use the localization sequence to record a short proof\footnote{This is similar to the ideas of Hindry--Pacheco and Kahn in \cite[\S\S3.2-3.3]{MR2562456}.} of the Tate--Shioda relation (Corollary \ref{shioda}).  In turn, this gives a quick calculation (\ref{rm4}) of the height pairing $\Delta_{\ar}(\NS(X))$ on the N\'eron--Severi group of $X$. The same calculation in \cite{MR528839, MR3858404} requires a detailed analysis of various subgroups of $\NS(X)$. A beautiful introduction to these results is \cite{MR3586808}; see  \cite{SL1983, MR2135283, MR4032302} for Weil-\'etale analogues. 

The second proof (\S \ref{suzuki}) of Theorem \ref{main} uses only (\ref{raynaud})
and the Weil-\'etale formulations of the two conjectures. 
In this proof, we do not compare each term of the two special value formulas
and entirely work in derived categories.

\subsection*{Notations} Throughout, $k=\mathbb F_q$ is a finite field of characteristic $p$ and $T =\mathrm{Spec}~k$; if $\bar{k}$ is an algebraic closure of $k$, let $\bar{T} = \mathrm{Spec}~\bar{k}$. The function field of $S$ is $F = k(S)$.  Let $X$ be a smooth proper surface over $T$ with a flat proper morphism $\pi:X \to S$ with smooth geometrically connected generic fiber $X_0$ over Spec~$F$.  The Jacobian $J$ of $X_0$ is an Abelian variety over $F$. 

\subsection{The Artin--Tate conjecture} Let $k=\mathbb F_q$ and $F =k(S)$.  For any scheme $V$ of finite type over $T$, the zeta function $\zeta(V,s)$ is defined as \[ \zeta(V,s) = \prod_{v \in V} \frac{1}{(1-q_v^{-s})};\]the product is over all closed points $v$ of $V$ and $q_v$ is the size of the finite residue field $k(v)$ of $v$. If $V$ is smooth proper (geometrically connected) of dimension $d$, then the zeta function $\zeta(V,s)$ factorizes as 
 \[\zeta(V,s) = \frac{P_1(V,q^{-s}) \cdots P_{2d-1}(V, q^{-s})}{P_0(V,q^{-s}) \cdots P_{2d}(V, q^{-s})},\quad P_0 = (1-q^{-s}),\quad P_{2d} = (1-q^{d-s}),\]
where $P_i(V,t)\in \mathbb Z[t]$ is the characteristic polynomial of Frobenius acting on the $\ell$-adic \'etale cohomology $H^i(V \times_{T} \bar{T}, \mathbb Q_{\ell})$ for any prime $\ell$ not dividing $q$; by Grothendieck and Deligne,  $P_j(V,t)$ is independent of  $\ell$. One has the factorization \cite[(4.1)]{Tate1966}  (the second equality uses Poincar\'e duality)
\begin{equation}\label{zetaX} \zeta(X,s) = \frac{P_1(X,q^{-s})\cdot P_3(X,q^{-s})}{(1-q^{-s})\cdot P_2(X,q^{-s})\cdot(1-q^{2-s})} = \frac{P_1(X,q^{-s})\cdot P_1(X,q^{1-s})}{(1-q^{-s})\cdot  P_2(X,q^{-s})\cdot(1-q^{2-s})}. \end{equation} 
Let $\rho(X)$ be the rank of the finitely generated N\'eron--Severi group $\NS(X)$.  The intersection $D\cdot E$ of divisors $D$ and $E$ provides a  symmetric non-degenerate bilinear pairing on $\NS(X)$; the height pairing $ \langle D,E\rangle_{\ar} $ \cite[Remark 3.11]{MR3858404} on $\NS(X)$ is related to the intersection pairing as follows:
\[\NS(X) \times \NS(X) \to  \mathbb Q(\log q), \qquad D,E\mapsto  \langle D,E\rangle_{\ar} = (D\cdot E) \log q.\]  
 Let $A$ be the reduced identity component $\Pic^{\red, 0}_{X/k}$ of the Picard scheme $\Pic_{X/k}$ of $X$.  Let 
 \begin{equation}\label{alphax} \alpha(X) = \chi(X, \mathcal O_X) -1+\dim(A).\end{equation} We write $[G]$ for the order of a finite group $G$. 
\begin{conjecture}[Artin--Tate \protect{\cite[Conjecture (C)]{Tate1966}}]\label{AT}
The Brauer group $\Br(X)$ is finite, $\ord_{s=1} P_2(X, q^{-s}) = \rho(X)$, 
and the special value  \[P^*_2(X, q^{-1}) : = \lim_{s\to 1} \frac{P_2(X, q^{-s})}{(s-1)^{\rho(X)}}\] of $P_2(X,t)$ at $t=1/q$ $($this corresponds to $s=1)$ satisfies 
\begin{equation}\label{eqn1} P^*_2(X, q^{-1}) = [\Br(X)]\cdot\Delta_{\ar}(\NS(X))\cdot q^{-\alpha(X)}.\end{equation} 
Here $\Delta_{\ar}(\NS(X))$ is the discriminant $($see \S \ref{pairings}\,$)$ of the height pairing on $\NS(X)$.
\end{conjecture}   

\begin{remark*} The discriminant $\Delta_{\ar}(\NS(X))$ of the height pairing on $\NS(X)$ is related to the discriminant $\Delta(\NS(X))$ of the intersection pairing as follows: $\Delta_{\ar}(\NS(X)) = \Delta(\NS(X))\cdot(\log q)^{\rho(X)}$. 
\end{remark*} 
\subsection{Discriminants}\label{pairings}
For more details on the basic notions recalled next, see~\cite[\S 2.8]{Yun} and~\cite{MR899399}. Let $N$ be a finitely generated Abelian group $N$ and  let $\psi: N \times N \to K$ be a symmetric bilinear form with values in any field $K$ of characteristic zero. If $\psi: N/{\tor} \times N/{\tor} \to K$ is non-degenerate, the discriminant $\Delta(N)$  is defined as the determinant of the matrix $\psi(b_i, b_j)$ divided by  $(N: N')^2$ where $N'$ is the subgroup of finite index generated by a maximal linearly independent subset $\{b_i\}$ of $N$. Note that $\Delta(N)$ is independent of the choice of the subset $\{b_i\}$ and the subgroup $N'$ and incorporates the order of the torsion subgroup of $N$. For us, $K = \mathbb Q$ or $\mathbb Q(\log q)$.
 
 Given a short exact sequence  $ 0 \to N' \to N \to N'' \to 0$
which splits over $\mathbb Q$ as an orthogonal direct sum $N_{\mathbb Q} \cong N'_{\mathbb Q} \oplus N''_{\mathbb Q}$
 with respect to a definite pairing $\psi$ on $N$, one has the following standard relation
\begin{equation}\label{deltann} \Delta(N) = \Delta(N')\cdot\Delta(N'').
\end{equation} 
 Given a map $f: C \to C'$ of Abelian groups with finite kernel and cokernel,  the invariant  $ z(f) = \frac{[\textrm{Ker}(f)]}{[\textrm{Coker}(f)]}$ \cite{Tate1966} extends to the derived category $\mathcal D$ of complexes in Abelian groups with bounded and finite homology: given any such complex $C_{\bullet}$, the invariant
 \[z(C_{\bullet}) = \prod_i[H_i(C_{\bullet})]^{(-1)^i}\]
 is an Euler characteristic; for any triangle $K \to L \to M \to K[1]$ in $\mathcal D$, the following relation holds
 \begin{equation}\label{euler-z}z(K)\cdot z(M) = z(L).\end{equation}
 One recovers $z(f)$ viewing $f:C \to C'$ as a complex in degrees zero and one. For any pairing $\psi: N \times N \to \mathbb Z$, the induced map $N \to R\textrm{Hom}(N,\mathbb Z)$ recovers $\Delta(N)$ above:
 \[\Delta(N) = z(N \to R\textrm{Hom}(N,\mathbb Z))^{-1}.\]
\qed

\subsection{The Birch--Swinnerton-Dyer conjecture}

For more details on the basic notions recalled next, see \cite{MR4032302}. Let $J$ be the Jacobian of $X_0$. Recall that the complete L-function \cite{SDPP_1969-1970__11_2_A4_0, MR330174}, \cite[\S 4]{MR4032302} of $J$ is defined as a product of local factors 
\begin{equation} L(J,s) = \prod_{v\in S} \frac{1}{L_v(J, q_v^{-s})}.
\end{equation} For any closed point $v$ of $S$, the local factor $L_v(J,t)$ is the characteristic polynomial of Frobenius on
  \begin{equation}\label{rm7} H^1_{\et}(J \times F_v^{\sep}, \mathbb Q_{\ell})^{I_v},\end{equation}
 where $F_v$ is the complete local field corresponding to $v$ and $I_v$ is the inertia group at $v$. By \cite[Proposition~4.1]{MR4032302}, $L_v(J,t)$ has coefficients in $\mathbb Z$ and is independent  of $\ell$, for any prime $\ell$ distinct from the characteristic of $k$. Let $\Sha(J/F)$ be the Tate--Shafarevich group of $J$ over Spec~$F$ and let $r$ be the rank of the finitely generated group $J(F)$. Let $\Delta_{\NT}(J(F))$ be the discriminant of the N\'eron--Tate pairing \cite[p.~419]{Tate1966}, \cite[\S 1.5]{MR2000469} on $J(F)$:
 \begin{equation}\label{NTdefinition}
 J(F) \times J(F) \to \mathbb Q(\log~q), \quad (\gamma, \kappa) \mapsto \langle \gamma, \kappa\rangle_{\NT}.
 \end{equation}   Let $\mathcal J \to S$ be the N\'eron model of $J$; for any closed point $v \in S$, define $c_v =[\Phi_v(k_v)]$ where $\Phi_v$ is the group of connected components of $\mathcal J_v$ and put $c(J) = \prod_{v \in S} c_v$; this is a finite product as $c_v=1$ for all but finitely many $v$.  
Let $\Lie\,\mathcal J$ be the locally free sheaf on $S$ defined by the  Lie algebra of $\mathcal J$. 
Recall the\footnote{By \cite[Corollary 4.5]{MR4032302}, this is equivalent to the formulation in \cite{Tate1966}.}
 \begin{conjecture}[Birch--Swinnerton-Dyer]\label{BSD} The group $\Sha(J/F)$ is finite, $\ord_{s=1}L(J,s) =r$, and the special value 
 \[L^*(J,1):=  \lim_{s\to 1} \frac{L(J,s)}{(s-1)^r}\] satisfies
 \begin{equation}\label{eqn2}L^*(J,1)  = [\Sha(J/F)]\cdot{\Delta_{\NT}(J(F))}\cdot c(J)\cdot q^{\chi(S, \Lie\,\mathcal J)}.\end{equation}
 \end{conjecture}

 The proof of Theorem \ref{main}, \emph{i.e.} the equivalence of Conjectures \ref{AT} and \ref{BSD}, naturally divides into four parts: 
 
 \begin{itemize}
 \item $\Br(X)$ is finite if and only if $\Sha(J/F)$ is finite. This is known \cite[(4.41), Corollaire (4.4)]{MR244271}. 
\item  Comparison of $\chi(S, \Lie\,\mathcal J)$ and $\alpha(X)$ given in (\ref{raynaud}). This is known \cite[p.~483]{MR2092767}. For the convenience of the reader, we recall it  in   \S \ref{comparison}.
\item (Proposition \ref{prop2}) $\ord_{s=1} P_2(X, q^{-s}) = \rho(X)$ if and only if $\ord_{s=1}L(J,s) =r$. 
 \item (\S \ref{core}) $P_2^*(X, 1)$ satisfies (\ref{eqn1}) if and only if $L^*(J,1)$ satisfies (\ref{eqn2}). 
\end{itemize}
 The first two parts are not difficult and we provide elementary proofs of the last two parts. 

\subsection*{Acknowledgements}  This paper would not exist without the inspiration provided by  \cite{flachsiebel, MR528839, MR3858404, Geisser1, Yun} in terms of both mathematical ideas and clear exposition.  We thank Professors Liu, Lorenzini and K.~Sato for their valuable comments on an earlier draft. We heartily thank the referee for a valuable and detailed report.

\section{Preparations}

 \subsection{Elementary identities and known results} 
 The N\'eron--Severi group $\NS(X)$ is the group of $k$-points of the group scheme $\NS_{X/k} = \pi_0(\Pic_{X/k})$ of connected components of the Picard scheme $\Pic_{X/k}$ of $X$. Let $A = \Pic^{\red, 0}_{X/k}$.  The Leray spectral sequence for the morphism $X \to \mathrm{Spec}~k$ and the \'etale sheaf $\mathbb G_m$ provides the first exact 
 sequences \cite[Proposition~4, p.~204]{BLR} below:
 \[0\To  \Pic(k) \To \Pic(X) \To \Pic_{X/k}(k) \To \Br(k) \quad\text{and}\quad  0 \To \Pic^0_{X/k} \To \Pic_{X/k} \To \pi_0(\Pic_{X/k}) \To 0.\] 
Since $\Br(k)=0$, $H^1_{\et}(\mathrm{Spec}~k, \Pic^0_{X/k}) = H^1_{\et}(\mathrm{Spec}~k, \Pic^{\red, 0}_{X/k})$ and $ H^1_{\et}(\mathrm{Spec}~k, A) =0$ (Lang's theorem \cite[p.~209]{Tate1966}), this provides 
\begin{equation}\label{eq7} \Pic_{X/k}(k) = \Pic(X) \quad\text{and}\quad \NS(X) = \NS_{X/k}(k) = \frac{\Pic(X)}{A(k)}.\end{equation}

Let $P$ be the identity component of the Picard scheme $\Pic_{S/k}$ of $S$. Let $B$ be the cokernel of the natural injective map $\pi^*: P \to A$. So one has short exact sequences (using Lang's theorem \cite[p.~209]{Tate1966} for the last sequence)
\begin{equation}\label{eq5} 
A = \Pic^{\red, 0}_{X/k}, \quad P= \Pic^0_{S/k},\quad 0 \To P \To A\To B\To 0, \quad \text{and}\quad 0 \To P(k) \To A(k) \To B(k) \To 0.
\end{equation} 
It is known that \cite[p.~428]{Tate1966}
\begin{equation}\label{picard}
P_1(S, q^{-s}) =P_1(P, q^{-s}), \quad P_1(X, q^{-s}) = P_1(A, q^{-s}),\quad\text{and}\quad P_1(A, q^{-s})= P_1(P, q^{-s})\cdot P_1(B, q^{-s}).
\end{equation} 
For any Abelian variety $G$ of dimension $d$ over $k= \mathbb F_q$, it is well known that \cite[p.~429, top line]{Tate1966} (or \cite[6.1.3]{MR528839})
\begin{equation}\label{fq-points} P_1(G, 1) = [G(k)] \quad\text{and}\quad P_1(G, q^{-1}) = [G(k)]q^{-d}.\end{equation} 

\subsection{ Comparison of $\chi(S, \Lie\,\mathcal J)$ and $\alpha(X)$}\label{comparison}  
 It is known \cite[p.~483]{MR2092767} that
\begin{equation}\label{raynaud}
\chi(S, \Lie\,\mathcal J)  - \dim(B) = -\alpha(X).
\end{equation} 
We include their proof here for the convenience of the reader. A special case of this is due to Gordon \cite[Proposition 6.5]{MR528839}. The Leray spectral sequence for $\pi$ and $\mathcal O_X$ provides $H^0(S, \mathcal O_S) \cong H^0(X, \mathcal O_X)$, 
\[ 0 \to H^1(S, \mathcal O_S) \to H^1(X, \mathcal O_X) \to H^0(S, R^1\pi_*\mathcal O_X) \to 0, \quad H^2(X, \mathcal O_X) \cong H^1(S, R^1\pi_*\mathcal O_X).\]
 This proves $\chi(X, \mathcal O_X) = \chi(S, \mathcal O_S) - \chi(S, R^1\pi_*\mathcal O_X)$. 
 Recall that $\mathcal J$ is the N\'eron model of the Jacobian $J$ of $X_0$.  As  the kernel and cokernel of the natural map\footnote{The map $\phi$ is obtained by the composition of the maps $R^1\pi_*\mathcal O_X \to \Lie\,P$  \cite[Proposition 1.3 (b)]{MR2092767} and $\Lie\,P \to \Lie\,Q$ \cite[Theorem 3.1]{MR2092767} with  $Q \xrightarrow{\sim} \mathcal J$ \cite[Facts 3.7 (a)]{MR2092767}; it uses the fact that $X$ is regular, $\pi:X \to S$ is proper flat, and $\pi_*\mathcal O_X = \mathcal O_S$.} $\phi: R^1\pi_*\mathcal O_X \to \Lie\,\mathcal J$
are torsion sheaves on $S$ of the same length  \cite[Theorem 4.2]{MR2092767}, we have  \cite[p.~483]{MR2092767}
\begin{equation}\label{equality}
\chi(S, R^1\pi_*\mathcal O_X) = \chi(S, \Lie\,\mathcal J).
\end{equation}
Thus, 
\begin{align*} 
\alpha(X) &\overset{(\ref{alphax})}{=} \chi(X, \mathcal O_X) -1+\dim(A) = \chi(S, \mathcal O_S) - \chi(S, R^1\pi_*\mathcal O_X) -1+\dim(A) \\
&=  1 - \dim(P)  - \chi(S, \Lie\,\mathcal J) -1 + \dim(A)=  - \chi(S, \Lie\,\mathcal J) + \dim(A) - \dim(P)\\
&\overset{(\ref{eq5})}{=} - \chi(S, \Lie\,\mathcal J) + \dim(B).
\end{align*}

\subsection{The Tate--Shioda relation about the N\'eron--Severi group} The structure of $\NS(X)$ depends on the singular fibers of the morphism $\pi:X \to S$. 
\subsubsection{Singular fibers}\label{singular}  Let $Z =\{ v \in S~|~\pi^{-1}(v) =X_v~\textrm{is~not~smooth}\}$. For any $v\in S$, let $G_v$ be the set of irreducible components $\Gamma_i$ of $X_v$, let $m_v$ be the cardinality of $G_v$, and $m: = \sum_{v \in Z} (m_v-1)$; for any $i \in G_v$, let $r_i$ be the number of irreducible components of $\Gamma_i \times \overline{k(v)}$. Let $R_v$ be the quotient 
\begin{equation}\label{Rv} R_v = \frac{\mathbb Z^{G_v}}{\mathbb Z}\end{equation}of the free Abelian group generated by the irreducible components of $X_v$ by the subgroup generated by the cycle associated with $X_v = \pi^{-1}(v)$. If $v\notin Z$, then $R_v$ is trivial.

Let $U = S - Z$; the map $X_U = \pi^{-1}(U) \to U$ is smooth.  For any finite $Z' \subset S$ with $Z \subset Z'$, we consider $U' = S - Z'$ and $X_{U'} = X - \pi^{-1}(U')$.  
The following proposition provides a description of $\NS(X) \overset{(\ref{eq7})}{\cong} \Pic(X)/{A(k)}$. 
\begin{proposition}\label{lastlemma}
\leavevmode
\begin{enumerate}[label=(\roman*)]
\item The natural maps $\pi^*: \Pic(S) \to  \Pic(X)$ and $\pi^*: \Pic(U') \to \Pic(X_{U'})$ are injective.
\item There is an exact sequence
\begin{equation}\label{anotherses} 0 \To  \underset{v\in Z}{\oplus}~R_v \To \frac{\Pic(X)}{\pi^*\Pic(S)} \To \Pic( X_0) \To 0.
\end{equation}
\end{enumerate}
\end{proposition} 
\begin{proof}  (i) 
 From the Leray spectral sequence for $\pi: X\to S$ and the \'etale sheaf $\mathbb G_m$ on $X$, we get the exact sequence
 \[ 0 \To H^1_{et}(S, \pi_* \mathbb G_m) \To H^1_{et}(X, \mathbb G_m) \To H^0(S, R^1\pi_*\mathbb G_m) \To \Br(S).\] 
Now $X_0$ being geometrically connected and smooth over $F$ implies \cite[Remark 1.7a]{MR656050} that $\pi_*\mathbb G_m$ is the sheaf $\mathbb G_m$ on $S$. 
This provides the injectivity of the first map. The same argument with $U'$ in place of $S$ provides the injectivity of the second.

(ii) The class group $\textrm{Cl}(Y)$ and the Picard group $\Pic(Y)$ are isomorphic for regular schemes $Y$ such as $S$ and $X$. The localization sequences for $X_{U'} \subset X$ and $U' \subset S$ can be combined as

\[\begin{tikzcd}
0 \arrow[r,] &\Gamma({S}, \mathbb G_m)   \arrow[r, ] \isoarrow{d} &\Gamma({U'}, \mathbb G_m) \arrow[r,] \isoarrow{d} &\underset{v\in Z'}{\oplus} \mathbb Z \arrow[r, ] \arrow[d, hook, red] & \Pic(S)   \arrow[r, ] \arrow[d, hook, red] &  \Pic(U') \arrow[r,] \arrow[d, hook, red]& 0\\
0 \arrow[r,] & \Gamma({X}, \mathbb G_m)  \arrow[r,] &\Gamma(X_{U'}, \mathbb G_m) \arrow[r,]  &\underset{v\in Z'}{\oplus} \mathbb Z^{G_v} \arrow[r,]  & \Pic(X)   
\arrow[r,]   &  \Pic(X_{U'}) \arrow[r,]  & 0.
\end{tikzcd}\]
Here $\Gamma ({X}, \mathbb G_m) = H^0_{et}({X}, \mathbb G_m) = H^0_{Zar}({X}, \mathbb G_m)$. The induced exact sequence on the cokernels of the vertical maps is
\[0 \To  \underset{v\in Z'}{\oplus} R_v \To \frac{\Pic(X)}{\pi^*\Pic(S)} \To \frac{\Pic(X_{U'})}{\pi^*\Pic(U')} \To 0.\]
In particular, we get this sequence for $Z$ and $U$. By assumption, $X_v$ is geometrically irreducible for any $v \notin Z$; so $R_v =0$ for any $v\notin Z$. 
So this means that, for any $U' = S - Z'$ contained in $U$, the induced maps 
\[\frac{\Pic(X_U)}{\pi^*\Pic(U)}\To \frac{\Pic(X_{U'})}{\pi^*\Pic(U')}\]
are isomorphisms. Taking the limit over $Z'$ gives us the exact sequence in the proposition.\end{proof} 

\begin{corollary}\label{shioda}
\leavevmode
\begin{enumerate}[label=(\roman*)]
\item The Tate--Shioda relation \cite[(4.5)]{Tate1966} $\rho(X) = 2 + r + m$  holds. 
\item One has an exact sequence 
\[0 \To B(k) \To \frac{\Pic(X)}{\pi^*\Pic(S)} \To \frac{\NS(X)}{\pi^*\NS(S)} \To 0.\]
\end{enumerate}
\end{corollary} 
\begin{proof}  
(i) Since $r$ is the rank of $J(F)$, the rank of $\Pic(X_0)$ is $r+1$. Since $\Pic(S)$ has rank one, $A(k)$ is finite and $m = \sum_{v \in Z} (m_v-1)$, this follows from (\ref{eq7}) and (\ref{anotherses}). 

\noindent (ii) This follows from the diagram
\[
\begin{tikzcd}
0 \arrow[r,] & P(k)  \arrow[r,"\pi^*"] \arrow[d, hook, red] & A(k)  \arrow[r,] \arrow[d, hook, red] &B(k) \arrow[r,] \arrow[d, hook, red] & 0 \\
0 \arrow[r,] &\Pic(S) \arrow[r,"\pi^*"]  \arrow[d, ] & \Pic(X) \arrow[r,] \arrow[d, ] &\frac{\Pic(X)}{\pi^*\Pic(S)} \arrow[r,]  \arrow[d, ]& 0\\
0 \arrow[r,] & \NS(S) \arrow[r, "\pi^*"]  & \NS(X) \arrow[r,] & \frac{\NS(X)}{\pi^*\NS(S)} \arrow[r,] & 0.
\end{tikzcd}\]
\end{proof}

 \subsection{Relating the order of vanishing at $s=1$ of $P_2(X, q^{-s})$ and $L(J,s)$} 
 
 By\footnote{This proposition, first stated on Page 176 of \cite{MR528839}, has a typo in the formula for $P_2$ which is corrected in its restatement on Page 193. We only need the part  about $P_2$ (and this is elementary).} \cite[Proposition 3.3]{MR528839}, one has  
\begin{equation}\label{newone}
\zeta(X_v,s) = \frac{P_1(X_v, q_v^{-s})}{(1-q_v^{-s})\cdot P_2(X_v,q_v^{-s})}, \quad\text{and}\quad P_2(X_v,q_v^{-s}) =  \left\{\begin{array}{lr}
        (1-q_v^{1-s}),  & \text{for } v \notin Z\\
        \prod_{i\in G_v}(1-(q_v)^{r_i(1-s)}), & \text{for } v\in Z
        \end{array}\right\},
\end{equation}
see \S \ref{singular} for notation. Using
\[Q_2(s) = \prod_{v\in Z} \frac{P_2(X_v, q^{-s} )}{(1-q_v^{1-s})},\quad \zeta (S,s) =\frac{P_1(S,q^{-s})}{(1-q^{-s})\cdot (1-q^{1-s})},\quad\text{and}\quad Q_1(s) = \prod_{v\in S} P_1(X_v,q_v^{-s}),\]
we can rewrite
 \[ \zeta(X,s) = \prod_{v \in S} \zeta(X_v,s) = \frac{1}{Q_2(s)}\cdot \prod_{v\in S} \frac{P_1(X_v,q_v^{-s})}{(1-q_v^{-s})\cdot (1-q_v^{1-s})} = \frac{\zeta(S,s)\cdot \zeta(S, s-1)\cdot Q_1(s)}{Q_2(s)}.\] 
 The precise relation between $P_2(X, q^{-s})$ and $L(J,s)$ is given by (\ref{eq9}).
 \begin{proposition} One has $\ord_{s=1} Q_2(s) =m$ and 
 \begin{equation}\label{rm2}  Q_2^*(1) = \lim_{s \to 1} \frac{Q_2(s)}{(s-1)^m} =\prod_{v\in Z}~ {\left( (\log q_v)^{(m_v-1)}\cdot \prod_{i\in G_v} r_i\right)},\end{equation}
 \begin{equation}\label{eq9} \frac{P_2(X,q^{-s}) }{(1-q^{1-s})^2}= P_1(B,q^{-s})\cdot P_1(B,q^{1-s})\cdot L(J,s)\cdot Q_2(s).\end{equation}
 \end{proposition} 
 \begin{proof} Observe that (\ref{rm2}) is elementary: for any positive integer $r$,
 one has
 \[ \lim _{s \to 1} \frac{(1-q_v^{r(1- s)})}{(s-1)} = \lim _{s \to 1} \frac{(1-q_v^{r(1- s)})}{(1-q_v^{1-s})}\cdot \frac{(1-q_v^{1-s})}{(s-1)}  = \lim _{s \to 1} (1 + q_v^{1-s} + \cdots+ q_v^{(r-1)(1-s)})\cdot \log q_v = r\cdot \log q_v. \]
 For each $v \in Z$, this shows that 
 \[ \lim_{s \to 1} \frac{P_2(X_v, q^{-s})}{(s-1)^{m_v}} =  (\log q_v)^{m_v} \cdot \prod_{i\in G_v} r_i. \] 
 Therefore, we obtain that 
 \[  \lim_{s \to 1} \frac{Q_2(s)}{(s-1)^m}  = \prod_{v\in Z} \lim_{s \to 1}  \frac{\frac{P_2(X_v, q^{-s})}{(1-q_v^{1-s})}}{(s-1)^{m_v-1}} =\prod_{v \in Z} \lim_{s \to 1}  \frac{\frac{P_2(X_v, q^{-s})}{(s-1)^{m_v}}}{\frac{(1-q_v^{1-s})}{s-1}} = \prod_{v\in Z}  {\left( \frac{(\log q_v)^{m_v} \cdot \prod_{i\in G_v} r_i.}{\log q_v} \right)}. \]
   
 We now prove (\ref{eq9}).  Simplifying the identity
 \[ \frac{P_1(X,q^{-s})~\cdot~P_1(X,q^{1-s})}{(1-q^{-s})\cdot P_2(X,q^{-s})\cdot (1-q^{2-s})}=  \zeta(X,s) = \frac{P_1(S,q^{-s})}{(1-q^{-s})\cdot (1-q^{1-s})}\cdot \frac{P_1(S,q^{1-s})}{(1-q^{1-s})\cdot (1-q^{2-s})}\cdot \frac{Q_1(s)}{Q_2(s)} \]
 from (\ref{zetaX}) using (\ref{picard}), one obtains
 \[\frac{P_1(B,q^{-s})\cdot P_1(B,q^{1-s})}{P_2(X,q^{-s})}= \frac{1}{(1-q^{1-s})}\cdot \frac{1}{(1-q^{1-s})}\cdot \frac{Q_1(s)}{Q_2(s)}. \]
 On reordering, this becomes
 \[\frac{P_2(X,q^{-s}) }{(1-q^{1-s})^2}= \frac{P_1(B,q^{-s})\cdot P_1(B,q^{1-s})\cdot Q_2(s)}{Q_1(s)}.\]
 Let $T_{\ell}J$ be the $\ell$-adic Tate module of the Jacobian $J$ of $X$. For any $v \in S$, the Kummer sequence on $X$ and $J$ provides a $\textrm{Gal}(F_v^{\sep}/{F_v})$-equivariant isomorphism 
\[H^1_{\et}(X \times_S F_v^{\sep}, \mathbb Z_{\ell}(1)) \stackrel{\sim}{\To} T_{\ell}J \stackrel{\sim}{\longleftarrow} H^1_{\et}(J \times_F F_v^{\sep}, \mathbb Z_{\ell}(1)),\]
as $J$ is a self-dual Abelian variety: this provides the isomorphisms
 \[H^1_{\et}(J \times_F F_v^{\sep}, \mathbb Q_{\ell}) \cong H^1_{\et}(X 
\times_S F_v^{\sep}, \mathbb Q_{\ell}), \quad  H^1_{\et}(J \times_F F_v^{\sep}, \mathbb Q_{\ell})^{I_v} \cong H^1_{\et}(X \times_S F_v^{\sep}, \mathbb Q_{\ell})^{I_v}.\] 
From \cite[Th\'eor\`eme 3.6.1, pp.213--214]{MR601520} (the arithmetic case is in \cite[Lemma 1.2]{MR899399}), we obtain an isomorphism 
\[H^1_{\et}(X_v \times_{k(v)} \overline{k(v)}, \mathbb Q_{\ell}) \stackrel{\sim}{\longrightarrow} H^1_{\et}(X \times_S F_v^{\sep}, \mathbb Q_{\ell})^{I_v}.\] 
The definition of $L_v(J,t)$ in (\ref{rm7}) now implies that $P_1(X_v,q_v^{-s}) = L_v(J, q_v^{-s})$ and hence $Q_1(s)\cdot L(J,s) =1$.\end{proof}

\begin{proposition}\label{prop2}
\leavevmode
\begin{enumerate}[label=(\roman*)]
\item $\ord_{s=1} P_2(X, q^{-s}) = \rho(X)$ if and only if $\ord_{s=1}L(J,s) =r$. 
\item One has 
\begin{equation}\label{eq1}
P_2^*(X,\frac{1}{q}) = P_1(B,q^{-1})\cdot P_1(B,1)\cdot L^*(J,1)\cdot Q_2^*(1)\cdot (\log~q)^2 \overset{(\ref{fq-points})}{=} \frac{[B(k)]^2}{q^{\dim(B)}}\cdot L^*(J,1)\cdot Q_2^*(1) \cdot (\log~q)^2.
\end{equation} 
\end{enumerate}
\end{proposition} 
\begin{proof} As  $P_1(B, q^{-s})\cdot P_1(B, q^{1-s})$ does not vanish at $s=1$ by  (\ref{fq-points}), it follows from (\ref{eq9}) that \[\ord_{s=1} P_2(X, q^{-s})  -2 = \ord_{s=1}L(J,s) + \ord_{s=1}Q_2(s).\] 
Corollary \ref{shioda} says $\rho(X) = r +m +2$; (i) follows as $\ord_{s=1} Q_2(s) =m$.  

For (ii), use (\ref{fq-points}) and (\ref{eq9}). \end{proof}

\subsection{Pairings on $\NS(X)$} Our next task is to compute $\Delta(\NS(X))$. 
\begin{definition}
\leavevmode
\begin{enumerate}[label=(\roman*)]
\item Let $\Pic^0(X_0)$ be the kernel of the degree map $\textrm{deg}: \Pic(X_0) \to \mathbb Z$; the order $\delta$ of its cokernel is, by definition, the index of $X_0$ over $F$.
\item Let $\alpha$ be the order of the cokernel of the natural map $\Pic^0(X_0) \hookrightarrow J(F)$. 
\item Let $H$ (horizontal divisor on $X$) be the Zariski closure in $X$ of  a divisor $d$ on $X_0$, rational over $F$, of degree $\delta$. 
\item The (vertical) divisor $V$ on $X$ is $\pi^{-1}(s)$ for a divisor $s$ of degree one on $S$. Such a divisor $s$ exists as $k$ is a finite field and so the index of the curve $S$ over $k$ is one. Writing $s = \sum a_i v_i$ as a sum of closed points $v_i$ on $S$ gives $V  =\sum a_i \pi^{-1}(v_i)$. Note that $V$ generates $\pi^*\NS(S) \subset \NS(X)$.
\end{enumerate}
\end{definition}

\begin{remark*} The definitions show that the intersections of the divisor classes $H$ and $V$ in $\NS(X)$ are given by
\begin{equation}\label{hv}
H\cdot V = \delta = V\cdot H \quad\text{and}\quad V\cdot V=0.
\end{equation} 
Also, since $\pi:X \to S$ is a flat map between smooth schemes, the map $\pi^*:CH(S) \to CH(X)$ on Chow groups is compatible with intersection of cycles. Since $V = \pi^*(s)$ and the intersection $s\cdot s = 0$ in $CH(S)$, one has  $V\cdot V =0$. 
\end{remark*}

Let $\NS(X)_0= (\pi^*\NS(S))^{\perp}$; as $V$ generates $\pi^*\NS(S)$, we see that $\NS(X)_0$  is the subgroup of divisor classes $Y$ such that  $Y\cdot X_v =0$ for any fiber $\pi^{-1}(v) =X_v$ of $\pi$; let $\Pic(X)_0$ be the inverse image of $\NS(X)_0$ under the projection $\Pic(X) \to \NS(X) \cong \frac{\Pic(X)}{A(k)}$.  
\begin{lemma}\label{gordon} $\NS(X)_0$ is the subgroup of $\NS(X)$ generated by divisor classes whose restriction to $X_0$ is trivial.
\end{lemma}
\begin{proof}  We need to show that $\NS(X)_0$ is equal to $K:= \textrm{Ker}(\NS(X) \to \NS(X_0))$. If $D$ is a vertical divisor ($\pi(D)\subset S$ is finite), then $D$ is clearly in $K$; by \cite[\S 9.1, Proposition 1.21]{MR1917232}, $D$ is in $\NS(X)_0$. 

 If $D$ has no vertical components, then $D\cdot V = \textrm{deg}(D_0)$. To see this, clearly we may assume $D$ is reduced and irreducible (integral) and so flat over $S$. So $\mathcal O_D$ is locally free over $\mathcal O_S$ of constant degree $n$ since $S$ is connected. But then $\textrm{deg}(D_0)$ is equal to $n$ as is the integer $D\cdot V$.\end{proof}

\begin{lemma}  Let us denote
\[ R=  \underset{v\in Z}{\oplus} R_v \quad\text{and}\quad E= B(k) \cap R \subset  \frac{\Pic(X)_0}{\pi^*\Pic(S)}.\] One has the exact sequences
\begin{align}\label{split}
&0 \To R \To \frac{\Pic(X)_0}{\pi^*\Pic(S)} \To \Pic^0(X_0) \To 0, \quad\text{and}\notag \\
&0 \To \frac{R}{E} \To \frac{\NS(X)_0}{\pi^*\NS(S)} \To \frac{ \Pic^0(X_0)}{B(k)/E} \To 0 .
\end{align} 
\end{lemma} 
\begin{proof}   Lemma \ref{gordon} shows that $R \subset  \frac{\Pic(X)_0}{\pi^*\Pic(S)}$. As $A(k)$ is the kernel of the map $\Pic(X) \to \NS(X)$, it follows that $A(k) \subset \Pic(X)_0$. Thus, $B(k)$ is a subgroup of $ \frac{\Pic(X)_0}{\pi^*\Pic(S)}$. 
 
The first exact sequence follows from Lemma \ref{gordon}; the second one follows from Corollary \ref{shioda} (ii).  \end{proof} 

\begin{lemma}\label{product} One has the equality \[ \Delta_{\ar}\left(\frac{\NS(X)_0}{\pi^*\NS(S)}\right) = [B(k)]^2\cdot  \alpha^2\cdot \Delta_{\NT}(J(F))\cdot \prod_{v \in Z}~\Delta_{\ar}(R_v).\]\end{lemma} 
\begin{proof} The exact sequence (\ref{split}) splits orthogonally over $\mathbb Q$: for any divisor $\gamma$ representing an element of $\Pic(X_0)$, consider its Zariski closure $\bar{\gamma}$ in $X$. Since the intersection pairing on $R_v$ is negative-definite \cite[\S 9.1, Theorem 1.23]{MR1917232}, the linear map $R_v \to \mathbb Z$ defined by $\beta \mapsto \beta\cdot\bar{\gamma}$ is represented by a unique element \[\psi_v(\gamma)\in R_v\otimes\mathbb Q\subset \frac{\NS(X)_0}{\pi^*\NS(S)}\otimes\mathbb Q.\]   Thus, the element
 \[\tilde{\gamma}:= \bar{\gamma}  -\sum_{v\in Z} \psi_v(\gamma)\]
 is {\it good} in the sense of \cite[\S 5, p.~185]{MR528839}: by construction, the divisor $\tilde{\gamma}$ on $X$ intersects every irreducible component of every fiber of $\pi$ with multiplicity zero.  Fix $\gamma, \kappa \in \Pic^0(X_0)$:  viewing them as elements of $J(F)$, one  computes their Neron--Tate pairing (\ref{NTdefinition}); also, one can compute the height pairing of $\tilde{\gamma}$ and $\tilde{\kappa}$ in $\NS(X)$.   These two are related by the identity  \cite[p.~429]{Tate1966} \cite[Remark 3.11]{MR3858404} 
 \[\langle \gamma, \kappa\rangle_{\NT}  = - \langle\tilde{\gamma}, \tilde{\kappa}\rangle_{\ar} = - (\tilde{\gamma}\cdot \tilde{\kappa})\cdot \log q.\]
 This says that 
 \begin{equation}\label{equality}
 \Delta_{\ar} \left(\Pic^0(X_0)\right)= \Delta_{\NT}\left(\Pic^0(X_0)\right).
 \end{equation}  The map
\[\Pic^0(X_0)\otimes\mathbb Q \to \frac{\NS(X)_0}{\pi^*\NS(S)}\otimes\mathbb Q, \qquad \gamma \mapsto \tilde{\gamma}\]
provides an orthogonal splitting of (\ref{split})  (over $\mathbb Q$). So 
\begin{align*}
\Delta_{\ar}\left(\frac{\NS(X)_0}{\pi^*\NS(S)}\right)& \overset{(\ref{deltann})}{=} \Delta_{\ar}\left(\frac{\Pic^0(X_0)}{B(k)/E}\right)\cdot \Delta_{\ar}\left(\frac{R}{E}\right) = \frac{[B(k)]^2}{e^2}\cdot \Delta_{\ar}\left({\Pic^0(X_0)}\right)  \cdot e^2 \Delta_{\ar}(R)\\
&\overset{(\ref{equality})}{=} [B(k)]^2\cdot \Delta_{\NT}\left({\Pic^0(X_0)}\right) \cdot \Delta_{\ar}(R)
\end{align*}
 where $e= [E]$ as the size of $E$. As 
\begin{equation}\label{eq4} 
 \Delta_{\NT}(\Pic^0(X_0)) = \alpha^2\cdot \Delta_{\NT}(J(F)) \quad\text{and} \Delta_{\ar}(R) = \prod_{v\in Z} \Delta_{\ar}(R_v),
\end{equation} 
this proves the lemma. \end{proof} 

With Lemma \ref{product} at hand we are almost ready to compute $\Delta_{\ar}(\NS(X))$. As the intersection pairing on $\NS(X)$ is not definite (Hodge index theorem), we cannot apply (\ref{deltann}). Instead, we use a variant of a lemma of Z.~Yun \cite{Yun}.

\subsubsection{A lemma of Yun} Given a non-degenerate symmetric bilinear pairing $\Lambda \times \Lambda \to \mathbb Z$ on a finitely generated Abelian group $\Lambda$, an isotropic subgroup $\Gamma$, a subgroup $\Gamma'$ containing $\Gamma$ and with finite index in $\Gamma^{\perp}$,  let $\Lambda_0 = \frac{\Gamma'}{\Gamma}$.  We recall from \S \ref{pairings} that  $\Delta(\Lambda) = z(D)^{-1}$ where $D:= \Lambda \to R\textrm{Hom}(\Lambda, \mathbb Z)$ and $\Delta(\Lambda_0) =  {z(D_0)}^{-1}$ where $D_0:= \Lambda_0 \to R\textrm{Hom}(\Lambda_0, \mathbb Z)$. Let $\Delta$ be the discriminant of the induced non-degenerate pairing $\Gamma \times \frac{\Lambda}{\Gamma'} \to \mathbb Z$:
\[\Delta = \frac{1}{z(C)}= \frac{1}{z(C')}, \quad C:= \Gamma \to R\textrm{Hom}\left(\frac{\Lambda}{\Gamma'}, \mathbb Z\right),\quad\text{and}\quad C':= \frac{\Lambda}{\Gamma'} \to R\textrm{Hom}(\Gamma, \mathbb Z).\]

\begin{lemma}[\emph{cf.} \protect{\cite[Lemma 2.12]{Yun}}]\label{yun-lemma}
One has $\Delta(\Lambda) = \Delta(\Lambda_0)\cdot \Delta^2$. 
\end{lemma} 
\begin{proof} Applying (\ref{euler-z}) to the maps of triangles 
\[\begin{tikzcd}
\Gamma \arrow[r] \arrow[d] &\Lambda \arrow[r]\arrow[d] & \frac{\Lambda}{\Gamma} \arrow[r] \arrow[d] & \Gamma[1 \arrow[d]]\\
R\textrm{Hom} \left( \frac{\Lambda}{\Gamma'}, \mathbb Z \right) \arrow[r,] & R\textrm{Hom} (\Lambda ,\mathbb Z ) \arrow[r,]& R\textrm{Hom} (\Gamma', \mathbb Z)\arrow[r,] &R\textrm{Hom} \left( \frac{\Lambda}{\Gamma'}, \mathbb Z \right) [1]
\end{tikzcd} \]
and
\[\begin{tikzcd}
\frac{\Gamma'}{\Gamma} \arrow[r] \arrow[d] &\frac{\Lambda}{\Gamma} \arrow[r]\arrow[d] & \frac{\Lambda}{\Gamma'} \arrow[r] \arrow[d] & \frac{\Gamma'}{\Gamma}[1 \arrow[d]]\\
R\textrm{Hom} \left( \frac{\Gamma'}{\Gamma}, \mathbb Z \right) \arrow[r,] & R\textrm{Hom} (\Gamma' ,\mathbb Z ) \arrow[r,]& R\textrm{Hom} (\Gamma, \mathbb Z)\arrow[r,] &R\textrm{Hom} \left( \frac{\Gamma'}{\Gamma}, \mathbb Z \right) [1]
\end{tikzcd} \]
shows that $z(D)\cdot z(C)^{-1} = z(D_0)\cdot z(C')$.
\end{proof} 
We can finally compute $\Delta_{\ar}(\NS(X))$. 

\begin{proposition} The following relations hold
\[ \Delta_{\ar}(\NS(X)) = \delta^2\cdot \Delta_{\ar}\left(\frac{\NS(X)_0}{\pi^*\NS(S)}\right)\cdot (\log q)^2 \quad\text{and}\quad  \Delta(\NS(X)) = \delta^2\cdot \Delta\left(\frac{\NS(X)_0}{\pi^*\NS(S)}\right).\]
\end{proposition} 
\begin{proof} Let $\mathbb Z \cong \Gamma = \pi^*\NS(S) \subset \NS(X) =\Lambda$ with $\Gamma' = \NS(X)_0$ and $\Lambda_0 = \frac{\NS(X)_0}{\pi^*\NS(S)}$. Lemma \ref{gordon} implies that
\[ \frac{\Lambda}{\Gamma'} =\frac{\NS(X)}{\NS(X)_0} \cong \mathbb Z \quad\text{and}\quad  C = \Gamma \to \textrm{Hom}\left(\frac{\NS(X)}{\NS(X)_0}, \mathbb Z\right),\]with 
$C$ as in Lemma \ref{yun-lemma}. 
Now  (\ref{hv}) shows that $\pi^*\NS(S)$ is isotropic and $\Delta =\delta$. The result follows from Lemma \ref{yun-lemma}.
\end{proof} 
Combining the previous proposition with Lemma \ref{product} provides the identity 
\begin{equation}\label{rm4}   \Delta_{\ar}(\NS(X)) = \delta^2\cdot  [B(k)]^2\cdot  \alpha^2\cdot \Delta_{\NT}(J(F))\cdot \prod_{v \in Z}~\Delta_{\ar}(R_v) \cdot (\log q)^2.
\end{equation}

For $ v \in S$, we put $\delta_v$ and $\delta'_v$ for the (local) index and period of $X\times F_v$ over the local field $F_v$. 
\begin{theorem}\label{geisser} \cite[Theorem 1.1]{Geisser1} Assume that $\Br(X)$ is finite. The following equality holds:
\begin{equation}\label{eqn-g}
[\Br(X)] \alpha^2\delta^2 = [\Sha(J/F)] \prod_{v \in S} \delta'_v \delta_v.
\end{equation}
\end{theorem}
\begin{remark} Note that for $v \in U$, one has $\delta_v = 1 = \delta'_v$ \cite[p.~603]{MR3858404}, \cite[(74)]{flachsiebel} (for $\delta_v=1$), \cite[Proposition (4.1) (a)]{MR244271} ($\delta'_v$ divides $\delta_v$); the basic reason is that if $v\in U$, then $X_v$ has a rational divisor of degree one as $k(v)$ is finite; this divisor lifts to a rational divisor of degree one on $X \times F_v$ by smoothness of  $X_v$. Also, $c_v =1$ \cite[Theorem 1, \S 9.5 p.~264]{BLR}. 
So $c(J) :=\prod_{v\in S}c_v $ satisfies
\begin{equation}\label{rm1} c(J) = \prod_{v\in Z}c_v.\end{equation} 
 \end{remark} 

\begin{lemma} One has
\begin{equation}\label{rm3} c(J)\cdot Q_2^*(1) = \prod_{v \in Z} \delta_v\cdot \delta'_v\cdot \Delta_{\ar}(R_v).\end{equation}
\end{lemma} 
\begin{proof} By a result of Flach and Siebel \cite[Lemma 17]{flachsiebel} (using Raynaud's theorem \cite[Theorem 5.2]{MR528839} in \cite{MR1717533}), one has 
\[\Delta_{\ar}(R_v) = \frac{c_v}{\delta_v\cdot \delta'_v} \cdot (\log q_v)^{m_v-1}\cdot \prod_{i\in G_v} r_i.\]
 So we find that 
\begin{align*}
\prod_{v \in Z} \delta_v\cdot \delta'_v\cdot \Delta_{\ar}(R_v) &= \prod_{v\in Z}  \left({c_v}\cdot (\log q_v)^{m_v-1}\cdot \prod_{i\in G_v} r_i\right)=\prod_{v\in Z}  {c_v}\cdot \prod_{v\in Z} \left((\log q_v)^{m_v-1}\cdot \prod_{i\in G_v} r_i\right) \\
 &\overset{(\ref{rm1})}{=} c(J)\cdot \prod_{v\in Z}~\left((\log q_v)^{m_v-1}\cdot \prod_{i\in G_v} r_i \right) \overset{(\ref{rm2})}{=} c(J)\cdot Q_2^*(1).
\end{align*}
\end{proof}

\section{First proof of Theorem \ref{main}}\label{core}
\begin{proof}[Proof of Theorem \ref{main}]  By (\ref{rm4}) and (\ref{rm3}), we have 
\[  \Delta_{\ar}(\NS(X)) = \frac{\alpha^2\,\delta^2}{\prod_{v \in Z} \delta_v\cdot \delta'_v}\cdot \Delta_{\NT}(J(F))\cdot c(J)\cdot  [B(k)]^2\cdot Q_2^*(1) \cdot (\log q)^2.\]
From Theorem \ref{geisser}, we have 
\[  [\Br(X)] \cdot \Delta_{\ar}(\NS(X)) = [\Sha(J/F)] \cdot \Delta_{\NT}(J(F))\cdot c(J)\cdot  [B(k)]^2\cdot Q_2^*(1) \cdot (\log q)^2.\]
Further with (\ref{raynaud}), we obtain 
\[
[\Br(X)] \cdot \Delta_{\ar}(\NS(X))\cdot q^{-\alpha(X)} =  [\Sha(J/F)] \cdot \Delta_{\NT}(J(F))\cdot c(J)\cdot q^{\chi(S, \Lie\,\mathcal J) }~.~[B(k)]^2\cdot Q_2^*(1)\cdot q^{-\dim(B)}\cdot (\log q)^2.\]
On the other hand, recall (\ref{eq1})
\[ P_2^*(X,\frac{1}{q})=  L^*(J,1)\cdot [B(k)]^2\cdot Q_2^*(1)\cdot q^{-\dim(B)}\cdot (\log q)^2.\]
The ratio of the previous two equalities gives 
\[ \frac{P_2^*(X,\frac{1}{q})}{ [\Br(X)] \cdot \Delta_{\ar}(\NS(X))\cdot q^{-\alpha(X)} } = \frac{L^*(J,1)}{ [\Sha(J/F)] \cdot \Delta_{\NT}(J(F))\cdot c(J)\cdot q^{\chi(S, \Lie\,\mathcal J) }}. \]
This equality implies Theorem \ref{main}.  \end{proof}

\section{Second proof of Theorem \ref{main}}\label{suzuki}
We will give another more direct proof of Theorem \ref{main} using Weil-\'etale cohomology.
We refer the reader to \cite{MR2135283, Gei04, MR4032302}
for basics about Weil-\'etale cohomology over finite fields.
Throughout this section, we assume that $\Br(X)$ (and hence $\Sha(J / F)$) is finite.

\subsection{Setup}
Let $C \in D^{b}(T_{\et})$ be an object of the bounded derived category of
sheaves of Abelian groups on the small \'etale site $T_{\et}$.
Let $D \in D^{b}(\mathrm{FDVect}_{k})$ be an object of the bounded derived category of
finite-dimensional vector spaces over $k$.
Assume that the Weil-\'etale cohomology $H^{\ast}_{W}(T, C)$ is finitely generated
and the cohomology sheaf $H^{\ast}(C \otimes^{L} \Z / l \Z)$ is finite
in all degrees for all prime numbers $l \nmid q$.
Let $e \colon H^{i}_{W}(T, C) \to H^{i + 1}_{W}(T, C)$ be the map
defined by cup product with the arithmetic Frobenius $\in H^{1}_{W}(T, \Z)$.
It defines a complex
\[\cdots \stackrel{e}{\To} H^{i}_{W}(T, C) \stackrel{e}{\To} H^{i + 1}_{W}(T, C) \stackrel{e}{\To} \cdots\]
with finite cohomology.
Set $C_{\Q_{l}} = R \varprojlim_{n}(C \otimes^{L} \Z / l^{n} \Z) \otimes_{\Z_{l}} \Q_{l}$,
whose cohomologies are finite-dimensional vector spaces over $\Q_{l}$
(by the finiteness of $H^{\ast}(C \otimes^{L} \Z / l \Z)$)
equipped with an action of the geometric Frobenius $\varphi$ of $k$.
Define
\begin{align*}
				Z(C, t)
			&=
				\prod_{i}
				\det(1 - \varphi t \,|\, H^{i}(C_{\Q_{l}}))^{(-1)^{i + 1}},\\
				\rho(C)
			&=
				\sum_{j} (-1)^{j + 1} \cdot j \cdot \rank H^{j}_{W}(T, C),
		\\
				\chi_{W}(C)
			&=
				\chi(H^{\ast}_{W}(T, C), e),\quad\text{and}\\
				\chi(D)
			&=
				\sum_{j}
					(-1)^{j}
					\dim H^{j}(D).
	\end{align*}
Assume that $Z(C, t) \in \Q(t)$ and is independent of $l$.
Define $Q(C, D) \in \Q_{> 0}^{\times} \times (1 - t)^{\Z}$ to be
the leading term of the $(1 - t)$-adic expansion of the function
	\[
		\pm \frac{
			Z(C, t)
			(1 - t)^{\rho(C)}
		}{
			\chi_{W}(C) q^{\chi(D)}
		}
	\]
(the sign is the one that makes the coefficient positive).
It is the defect of a zeta value formula
of the form
	\[
			\lim_{t \to 1}
				Z(C, t)
				(1 - t)^{\rho(C)}
		=
			\pm \chi_{W}(C) q^{\chi(D)}.
	\]
We mention $Q(C, D)$ only when $H^{\ast}_{W}(T, C)$ is finitely generated,
$H^{\ast}(C \otimes^{L} \Z / l \Z)$ is finite and
$Z(C, t) \in \Q(t)$ is independent of $l$.
These conditions are satisfied for the cases of interest below.
We have
\[Q(C[1], D[1]) = Q(C, D)^{-1}.\]
If $(C, D)$, $(C', D')$ and $(C'', D'')$ are pairs as above,
and $C \to C' \to C'' \to C[1]$ and $D \to D' \to D'' \to D[1]$ are distinguished triangles,
then $Q(C', D') = Q(C, D) Q(C'', D'')$.

\subsection{Special cases}

We give two special cases of the above constructions.
First, let $\pi_{X} \colon X_{\et} \to T_{\et}$ be the structure morphism.
Let $P_{2}^{\diamond}(X, 1) (1 - t)^{\rho(X)'}$ be
the leading term of the $(1 - t)$-adic expansion of $P_{2}(X, t / q)$.

\begin{proposition}
	Let $(C, D) = (R \pi_{X, \ast} \Gm[-1], R \Gamma(X, \mathcal{O}_{X}))$.
	Then $H^{\ast}(C \otimes^{\mathrm{L}} \Z / l \Z)$ is finite, $H^{\ast}_{W}(T, C)$ is finitely generated, $Z(C, q^{-s}) = \zeta(X, s + 1)$ and
		\[
				Q(C, D)^{-1}
			=
				\frac{
						P_{2}^{\diamond}(X, 1)
					\cdot
						(1 - t)^{\rho(X)' - \rho(X)}
				}{
						[\Br(X)]
					\cdot
						\Delta(\NS(X))
					\cdot
						q^{- \alpha(X)}
				}.
		\]
	In particular, the statement $Q(C, D) = 1$ is equivalent to Conjecture \ref{AT}.
\end{proposition}

\begin{proof}
	We have $H^{\ast}_{W}(T, C) \cong H^{\ast}_{W}(X, \Gm[-1]) \cong H^{\ast}_{W}(X, \Z(1))$.
	The finiteness assumption on $\Br(X)$ implies the Tate conjecture for divisors on $X$
	and hence the finite generation of $H^{\ast}_{W}(X, \Z(1))$
	by \cite[Theorems~8.4 and~9.3]{Gei04}.
	The object $C \otimes^{\mathrm{L}} \Z / l \Z \cong R \pi_{X, \ast} \Z / l \Z(1) \in D^{b}(T_{\et})$ is constructible
	and hence its cohomologies are finite.
	We have $H^{i}(C_{\Q_{l}}) \cong R^{i} \pi_{X, \ast} \Q_{l}(1)$,
	which is the vector space $H_{\et}^{i}(X \times_{k} \Bar{k}, \Q_{l}(1))$ equipped with the natural Frobenius action.
	It follows that $Z(C, q^{-s}) = \zeta(X, s + 1)$.
	
	We calculate $Q(C, D)^{-1}$.
	By \eqref{zetaX}, \eqref{picard} and \eqref{fq-points},
	the leading term of the $(1 - t)$-adic expansion of $Z(C, t)$ is
		\begin{equation} \label{LeadZeta}
			- \frac{
				[A(k)]^{2}
			}{
				P_{2}^{\diamond}(X, 1) \cdot (q - 1)^{2} \cdot q^{\dim A - 1} \cdot (1 - t)^{\rho(X)'}
			}.
		\end{equation}
	By \cite[Theorems 7.5 and 9.1]{Gei04}, we have
		\[
				\chi_{W}(C)
			=
					\prod_{i}
						[H_{W}^{i}(X, \Z(1))_{\mathrm{tor}}]^{(-1)^{i}}
				\cdot
					R^{-1},
		\]
	where $R$ is the determinant of the pairing
		\[
				H_{W}^{2}(X, \Z(1))
			\times
				H_{W}^{2}(X, \Z(1))
			\stackrel{\cup}{\To}
				H_{W}^{4}(X, \Z(2))
			\To
				H_{\et}^{4}(X \times_{k} \Bar{k}, \Z(2))
			\cong
				\mathrm{CH}^{2}(X \times_{k} \Bar{k})
			\stackrel{\deg}{\longrightarrow}
				\Z.
		\]
	We have $H_{W}^{n}(X, \Z(1)) = 0$ for $n > 5$ by \cite[Theorem 7.3]{Gei04} and for $n < 1$ obviously.
	Also
		\[
					H_{W}^{1}(X, \Z(1))
				\cong
					k^{\times},
			\quad
					H_{W}^{2}(X, \Z(1))
				\cong
					\Pic(X),
			\quad\text{and}\quad 
					H_{W}^{3}(X, \Z(1))_{\mathrm{tor}}
				\cong
					\Br(X)
		\]
	by \cite[Proposition 7.4 (c) and (d)]{Gei04}.
	By \cite[Remark 3.3]{Gei18},
	the group $H_{W}^{i}(X, \Z(1))_{\mathrm{tor}}$ is Pontryagin dual to
	$H_{W}^{6 - i}(X, \Z(1))_{\mathrm{tor}}$ for any $i$.
	The above pairing defining $R$ can be identified with the intersection pairing
	$\Pic(X) \times \Pic(X) \to \Z$.
	Thus, with \eqref{eq7}, we have
		\begin{equation} \label{WetEuler}
				\chi_{W}(C)
			=
				\frac{
					[A(k)]^{2}
				}{
					[\Br(X)] \cdot \Delta(\NS(X)) \cdot (q - 1)^{2}
				}.
		\end{equation}
	Since the rank of $H_{W}^{i}(X, \Z(1))$ is $\rho(X)$ for $i = 2, 3$ and zero otherwise
	by \cite[Proposition 7.4 (c) and (d)]{Gei04},
	we have
		\begin{equation} \label{PicardNum}
			\rho(C) = \rho(X).
		\end{equation}
	Combining \eqref{alphax}, \eqref{LeadZeta}, \eqref{WetEuler} and \eqref{PicardNum},
	we get the desired formula for $Q(C, D)^{-1}$.
\end{proof}

Next, let $\pi_{S} \colon S_{\et} \to T_{\et}$ be the structure morphism.
Let $L^{\diamond}(J, 1) (1 - q^{- s})^{r'}$ be
the leading term of the $(1 - q^{- s})$-adic expansion of $L(J, s + 1)$.
Let $\Delta(J(F))$ be the discriminant of the pairing
$(\gamma, \kappa) \mapsto \langle \gamma, \kappa \rangle_{\NT} / \log q$ on $J(F)$.

\begin{proposition}
	Let $(C, D) = (R \pi_{S, \ast} \mathcal{J}[-1], R \Gamma(S, \Lie\,\mathcal{J}))$.
	Then $H^{\ast}(C \otimes^{\mathrm{L}} \Z / l \Z)$ is finite, $H^{\ast}_{W}(T, C)$ is finitely generated,
	$Z(C, q^{-s}) = L(J, s + 1)$ and
		\[
				Q(C, D)
			=
				\frac{
						L^{\diamond}(J, 1)
					\cdot
						(1 - t)^{r' - r}
				}{
						[\Sha(J / F)]
					\cdot
						\Delta(J(F))
					\cdot
						c(J)
					\cdot
						q^{\chi(S, \Lie\,\mathcal{J})}
				}.
		\]
	In particular, the statement $Q(C, D) = 1$ is equivalent to Conjecture \ref{BSD}.
\end{proposition}

\begin{proof}
	We have $H^{\ast}_{W}(T, C) \cong H_{W}^{\ast - 1}(S, \mathcal{J})$.
	The finiteness assumption of $\Sha(J / F)$ implies
	the finite generation of $H_{W}^{\ast}(S, \mathcal{J})$
	by \cite[Proposition 6.4]{MR4032302}.
	We have $C \otimes^{\mathrm{L}} \Z / l \Z \cong R \pi_{S, \ast}(\mathcal{J} \otimes^{L} \Z / l \Z)[-1]$.
	By the paragraph before the proof of \cite[Proposition 9.2]{MR4032302}
	and the first displayed equation in the proof of \cite[Proposition~9.2]{MR4032302},
	we know that $\mathcal{J} \otimes^{\mathrm{L}} \Z / l \Z \in D^{b}(S_{\et})$ is constructible.
	Hence $H^{\ast}(C \otimes^{\mathrm{L}} \Z / l \Z)$ is finite.
	We also have $H^{i}(C_{\Q_{l}}) \cong R^{i} \pi_{S, \ast} V_{l}(\mathcal{J})$
	(where $V_{l}$ denotes the $l$-adic Tate modules tensored with $\Q_{l}$),
	which is the vector space $H_{\et}^{i}(S \times_{k} \Bar{k}, V_{l}(\mathcal{J}))$
	equipped with the natural Frobenius action.
	Hence we have $Z(C, q^{-s}) = L(J, s + 1)$ by \cite[Satz 1]{Sch82}.
	We have
		\[
				\chi_{W}(C)
			=
					[\Sha(J / F)]
				\cdot
					\Delta(J(F))
				\cdot
					c(J)
		\]
	by \cite[Proposition 8.3]{MR4032302}.
	By \cite[Proposition 7.1]{MR4032302},
	the rank of $H_{W}^{i}(S, \mathcal{J})$ is $r$ for $i = 0, 1$ and zero otherwise.
	Hence $\rho(C) = - r$.
	The formula for $Q(C, D)$ follows.
\end{proof}

\subsection{Comparison}
Now Theorem \ref{main} follows from the following

\begin{proposition} \label{LeadingTerms}
	One has
		\[
				Q(R \pi_{X, \ast} \Gm[-1], R \Gamma(X, \mathcal{O}_{X}))^{-1}
			=
				Q(R \pi_{S, \ast} \mathcal{J}[-1], R \Gamma(S, \Lie\,\mathcal{J})).
		\]
\end{proposition}

\begin{proof}
	We have $R^{i} \pi_{\ast} \Gm = 0$ over $S_{\et}$ for all $i \ge 2$
	by \cite[Corollaire (3.2)]{MR244271}.
	Hence we have a distinguished triangle
		\[
				R \pi_{S, \ast} \Gm
			\To
				R \pi_{X, \ast} \Gm
			\To
				R \pi_{S, \ast} \Pic_{X / S}[-1]
			\To
				R \pi_{S, \ast} \Gm[1]
		\]
	in $D(T_{\et})$.%
	\footnote{
		Here $\Pic_{X / S} = R^{1} \pi_{\ast} \Gm$ is only an \'etale sheaf.
		The fppf sheaf denoted by the same symbol is not an algebraic space in general.
	}
	Similarly, we have a distinguished triangle
		\[
				R \Gamma(S, \mathcal{O}_{S})
			\To
				R \Gamma(X, \mathcal{O}_{X})
			\To
				R \Gamma(S, R^{1} \pi_{\ast} \mathcal{O}_{X})[-1]
			\To
				R \Gamma(S, \mathcal{O}_{S})[1].
		\]
	We have $Q(R \pi_{S, \ast} \Gm[-1], R \Gamma(S, \mathcal{O}_{S})) = 1$ 
	by the class number formula
	(\cite[Theorems 9.1 and 9.3]{Gei04}, or
	\cite[Theorems 5.4 and 7.4]{MR2135283} and the functional equation).
	Therefore
		\begin{equation} \label{XtoPicX}
				Q(R \pi_{X, \ast} \Gm[-1], R \Gamma(X, \mathcal{O}_{X}))^{-1}
			=
				Q(R \pi_{S, \ast} \Pic_{X / S}[-1], R \Gamma(S, R^{1} \pi_{\ast} \mathcal{O}_{X})).
		\end{equation}

	For a closed point $v \in S$,
	let $\iota_{v} \colon \Spec k(v) \hookrightarrow S$ be the inclusion.
	For any $i \in G_{v}$, let $k(v)_{i}$ be the algebraic closure of $k(v)$
	in the function field of $\Gamma_{i}$.
	Let $\iota_{v, i} \colon \Spec k(v)_{i} \to S$ be the natural morphism.
	Set
		\[
				E
			=
				\bigoplus_{v \in Z}
					\frac{
						\bigoplus_{i \in G_{v}}
							\iota_{v, i, \ast} \Z
					}{
						\iota_{v, \ast} \Z
					}.
		\]
	Let $j \colon \Spec F \hookrightarrow S$ be the inclusion.
	Then we have a natural exact sequence
		\[
				0
			\To
				E
			\To
				\Pic_{X / S}
			\To
				j_{\ast} \Pic_{X_{0} / F}
			\To
				0
		\]
	over $S_{\et}$ by \cite[Equations (4.10 bis) and (4.21)]{MR244271}
	(where the assumption \cite[Equation (4.13)]{MR244271} is satisfied
	since $k(v)$ is finite and hence perfect for all closed $v \in S$).
	Therefore we have a distinguished triangle
		\[
				R \pi_{S, \ast} E
			\To
				R \pi_{S, \ast} \Pic_{X / S}
			\To
				R \pi_{S, \ast} j_{\ast} \Pic_{X_{0} / F}
			\To
				R \pi_{S, \ast} E[1].
		\]
	Since $E$ is skyscraper,
	we have $Q(R \pi_{S, \ast} E, 0) = 1$
	by \cite[Theorem 3.1]{GS20b} (Step 3 of the proof is sufficient).
	Therefore
		\begin{equation} \label{PicXtoPicXzero}
				Q(R \pi_{S, \ast} \Pic_{X / S}[-1], R \Gamma(S, R^{1} \pi_{\ast} \mathcal{O}_{X}))
			=
				Q(R \pi_{S, \ast} j_{\ast} \Pic_{X_{0} / F}[-1], R \Gamma(S, R^{1} \pi_{\ast} \mathcal{O}_{X})).
		\end{equation}

	Applying $j_{\ast}$ to the exact sequence
		\[
				0
			\To
				J
			\To
				\Pic_{X_{0} / F}
			\To
				\Z
			\To
				0
		\]
	over $\Spec F_{\et}$,
	we obtain an exact sequence
		\[
				0
			\To
				\mathcal{J}
			\To
				j_{\ast} \Pic_{X_{0} / F}
			\To
				\Z
		\]
	over $S_{\et}$.
	Let $I$ be the image of the last morphism, so that we have an exact sequence
		\[
				0
			\To
				\mathcal{J}
			\To
				j_{\ast} \Pic_{X_{0} / F}
			\To
				I
			\To
				0.
		\]
	Then we have distinguished triangles
\begin{align*}
&R \pi_{S, \ast} \mathcal{J}
				\To
					R \pi_{S, \ast} j_{\ast} \Pic_{X_{0} / F}
				\To
					R \pi_{S, \ast} I
				\To
					R \pi_{S, \ast} \mathcal{J}[1],
			\quad\text{and}\\
			&R \pi_{S, \ast} I
				\To
					R \pi_{S, \ast} \Z
				\To
					R \pi_{S, \ast} (\Z / I)
				\To
					R \pi_{S, \ast} I[1].
\end{align*}
	We have $Q(R \pi_{S, \ast} \Z, 0) = 1$ again by the class number formula
	(\cite[Theorems 9.1 and 9.2]{Gei04} or \cite[Theorem 7.4]{MR2135283}).
	Since $\Z / I$ is skyscraper with finite stalks, we have $Q(R \pi_{S, \ast}(\Z / I), 0) = 1$
	by \cite[Theorem~3.1]{GS20b} (Step 2 of the proof is sufficient).
	Therefore
		\begin{equation} \label{PicXzeroToJ}
				Q(R \pi_{S, \ast} j_{\ast} \Pic_{X_{0} / F}[-1], R \Gamma(S, R^{1} \pi_{\ast} \mathcal{O}_{X}))
			=
				Q(R \pi_{S, \ast} \mathcal{J}[-1], R \Gamma(S, R^{1} \pi_{\ast} \mathcal{O}_{X})).
		\end{equation}
The complexes $R \Gamma(S, R^{1} \pi_{\ast} \mathcal{O}_{X})$ and $R \Gamma(S, \Lie\,\mathcal{J})$ have the same Euler characteristic by (\ref{equality}). Hence
		\begin{equation} \label{ROneToLie}
				Q(R \pi_{S, \ast} \mathcal{J}[-1], R \Gamma(S, R^{1} \pi_{\ast} \mathcal{O}_{X}))
			=
				Q(R \pi_{S, \ast} \mathcal{J}[-1], R \Gamma(S, \Lie\,\mathcal{J})).
		\end{equation}
Combining \eqref{XtoPicX}---\eqref{ROneToLie}, we get the desired equality.
\end{proof}

\subsection{A new proof of Geisser's formula}
The above proposition, combined with the results of the previous sections,
also gives a new proof of Theorem \ref{geisser} as follows.

\begin{proof}[Proof of Theorem \ref{geisser}]
	By Proposition \ref{LeadingTerms}, we have
		\[
				\frac{
					P_{2}^{\diamond}(X, 1)
				}{
						[\Br(X)]
					\cdot
						\Delta(\NS(X))
					\cdot
						q^{- \alpha(X)}
				}
			=
				\frac{
					L^{\diamond}(J, 1)
				}{
						[\Sha(J / F)]
					\cdot
						\Delta(J(F))
					\cdot
						c(J)
					\cdot
						q^{\chi(S, \Lie\,\mathcal{J})}
				}.
		\]
	By \eqref{eq1}, we have
		\[
				P_{2}^{\diamond}(X, 1)
			=
					L^{\diamond}(J, 1)
				\cdot
					q^{- \dim B}
				\cdot
					[B(k)]^{2}
				\cdot
					Q_{2}^{\diamond}(1),
		\]
	where $Q_{2}^{\diamond}(1)$ is the leading coefficient of the $(1 - q^{-s})$-adic expansion of $Q_{2}(s + 1)$.
	By \eqref{rm4} and \eqref{rm3}, we have
		\[
				\Delta(\NS(X))
			=
					\frac{
						\alpha^{2} \delta^{2}
					}{
						\prod_{v \in Z} \delta_{v}' \delta_{v}
					}
				\cdot
					\Delta(J(F))
				\cdot
					c(J)
				\cdot
					[B(k)]^{2}
				\cdot
					Q_{2}^{\diamond}(1).
		\]
	By \eqref{raynaud}, we have
		\[
				q^{- \alpha(X)}
			=
					q^{\chi(S, \Lie\,\mathcal{J})}
				\cdot
					q^{- \dim B}.
		\]
	Taking a suitable alternating product of these four equalities,
	we obtain \eqref{eqn-g}.
\end{proof}


\end{document}